\documentclass[11pt,letter]{amsart}
\usepackage{amsfonts,amssymb,amsthm,amsmath,amsxtra,amscd,verbatim,eucal}
\usepackage[all]{xy}
\usepackage[dvips]{graphics}
\usepackage{graphicx}

\theoremstyle{plain}
\newtheorem{thm}{Theorem}[section]

\newtheorem{prop}[thm]{Proposition}
\newtheorem{cor}[thm]{Corollary}

\theoremstyle{definition}

\theoremstyle{remark}

\newtheorem{rmk}[thm]{Remark}

\def\Z{{\mathbb Z}}

\def\C{{\mathbb C}}

\def\R{{\mathbb R}}
\def\Q{{\mathbb Q}}
\def\P{{\mathbb P}}

\def\O{\mathcal{O}}

\def\a{\alpha}

\def\ep{\epsilon}

\def\l{\lambda}

\def\r{\rho}

\def\D{\Delta}

\def\.{\cdot}
\def\~{\widetilde}
\def\^{\widehat}

\def\o{\circ}
\def\ov{\overline}
\def\rat{\dashrightarrow}

\renewcommand{\and}{ \ \, \text{and} \ \, }

\DeclareMathOperator{\cont} {cont}

\DeclareMathOperator{\NE} {NE}
\DeclareMathOperator{\NM} {NM}
\DeclareMathOperator{\CNE} {\ov{\NE}}
\DeclareMathOperator{\CNM} {\ov{\NM}}

\DeclareMathOperator{\PEff} {PEff}
\DeclareMathOperator{\Pic} {Pic}

\DeclareMathOperator{\Nef} {Nef}

\DeclareMathOperator{\Cl} {Cl}

\DeclareMathOperator{\Mov} {Mov}

\title{Rigidity properties of Fano varieties}

\author{Tommaso de Fernex}
\address{Department of Mathematics, University of Utah, 155 South 1400 East,
Salt Lake City, UT 48112-0090, USA}
\email{defernex@math.utah.edu}

\author{Christopher D. Hacon}
\address{Department of Mathematics, University of Utah, 155 South 1400 East,
Salt Lake City, UT 48112-0090, USA}
\email{hacon@math.utah.edu}

\thanks{The first author was partially supported by NSF CAREER Grant DMS-0847059.
The second author was supported by NSF Grant 0757897.}

\thanks{Version of October 30, 2009}

\subjclass[2000]{Primary 14J45; Secondary 14E22, 14D15}
\keywords{Fano varieties, Mori theory}

\begin{document}

\begin{abstract}
We overview some recent results on Fano varieties giving evidence
of their rigid nature under small deformations.
\end{abstract}

\maketitle

\section{Introduction}

From the point of view of the Minimal Model Program, Fano varieties
constitute the building blocks of uniruled varieties.
Important information on the biregular and birational geometry of a Fano variety
is encoded, via Mori theory,
in certain combinatorial data corresponding to the N\'eron--Severi
space of the variety.
It turns out that, even when there is actual variation in moduli,
much of such combinatorial data remains unaltered,
provided that the singularities are ``mild'' in an appropriate sense.
One should regard any statement of this sort as a rigidity property
of Fano varieties.

This paper gives an overview of Fano varieties, recalling
some of their most important properties and discussing
their rigid nature under small deformations.
We will keep a colloquial tone, referring the reader to the appropriate
references for many of the proofs. Our main purpose is indeed to give a broad
overview of some of the interesting features of this special class of varieties.
Throughout the paper, we work over the complex numbers.

\section{General properties of Fano varieties}

A {\it Fano manifold} is a projective manifold $X$ whose anticanonical
line bundle $-K_X := \wedge^n T_X$ is ample (here $n = \dim X$).

The simplest examples of a Fano manifolds are given by the projective spaces
$\P^n$. In this case, in fact, even the tangent space is ample (by a result of Mori,
we know that projective spaces
are the only manifolds with this property, cf. \cite{Mo79}).

In dimension two, Fano manifolds are known as Del Pezzo surfaces.
This class of surfaces has been widely studied in the literature
(it suffices to mention that several books have been written just on
cubic surfaces), and their geometry is quite well understood.
There are nine families of Del Pezzo surfaces.
The following theorem, obtained as a result of a
series of papers \cite{Nad90,Nad91,Cam91,Cam,KMM92a,KMM92},
shows that this is a general phenomenon.

\begin{thm}
For every $n$, there are only finitely many families
of Fano manifolds of dimension $n$.
\end{thm}

This theorem is based on the analysis of rational curves on Fano manifolds.
In this direction, we should also mention the following important result of
Campana and Koll\'ar--Miyaoka--Mori (cf. \cite{Cam, KMM92}).

\begin{thm}\label{t-rc}
Fano manifolds are rationally connected.
\end{thm}

Fano varieties arise naturally in the context of the Minimal Model Program.
This however leads us to work with possibly singular varieties.
The smallest class of singularities that one has to allow
is that of normal $\Q$-factorial varieties with terminal singularities.
However, one can enlarge the class of singularities further, and work
with normal $\Q$-Gorenstein varieties with log terminal (or, in some cases, even
log canonical) singularities.
In either case, one needs to consider $-K_X$ as a Weil divisor.
The hypothesis guarantees that some positive multiple $-mK_X$ is Cartier
(i.e. $\mathcal O _X(-mK_X)$ is a line bundle), so that one can
impose the condition of ampleness.

For us, a {\it Fano variety} will be a normal variety with $\Q$-Gorenstein log terminal
singularities such that $-K_X$ is ample. We will however be mostly interested
in the case where the singularities are $\Q$-factorial and terminal.

The above results are however more delicate in the singular case.
By a recent result of Zhang (cf. \cite{Z06}), it is known that Fano varieties
are rationally connected (see \cite{HM07} for a related statement).
Boundedness of Fano varieties is instead an open problem.
The example of a cone over a rational curve of degree $d$ shows that even for
surfaces, we must make some additional assumptions. In this example
one has that the minimal log discrepancies are given by $1/d$.
One may hope that if we bound the minimal log discrepancies away from $0$
the boundedness still holds. More precisely,
the BAB conjecture (due to Alexeev--Borisov--Borisov)
states that for every $n >0$ and any $\ep > 0$,
there are only finitely many families
of Fano varieties of dimension $n$ with $\ep$-log terminal singularities
(in particular, according to this conjecture, for every $n$ there are only finitely many
families of Fano varieties with canonical singularities).

Note that, by Theorem~\ref{t-rc},
it follows that Fano manifolds have the same cohomological invariants as
rational varieties (namely $h^i(\mathcal O _X)=h^{0}(\Omega ^q _X)=0$ for all $i,q>0$).
On the other hand, by celebrated results of Clemens--Griffiths and Iskovskikh--Manin,
it is known that there are
examples of Fano manifolds that are nonrational
(cf. \cite{IM71, CG72}). The search
for these examples was motivated
by the L\"uroth Problem. Note that it is still an open problem
to find examples of Fano manifolds that are not unirational.

Perhaps the most important result known to hold for Fano varieties
(for mild singularities and independently of their dimension), concerns the
combinatorial structure associated to the cone of effective curves.
The first instance of this
was discovered by Mori (cf. \cite{Mo82}) in the
smooth case. It is a particular case of the Cone Theorem (which holds for all
varieties with log terminal singularities).

\begin{thm}[Cone Theorem for Fano varieties]
The Mori cone of a Fano variety is rational polyhedral, generated
by classes of rational curves.
\end{thm}

Naturally one may also ask if there are similar results concerning the
structure of other cones of curves.
From a dual perspective, one would like to understand the structure of the various
cones of divisors on a Fano variety.
The strongest result along these lines was conjectured
by Hu--Keel in~\cite{HK} and recently proved by
Birkar--Cascini--Hacon--M$^{\rm c}$Kernan in \cite{BCHM}:

\begin{thm}\label{t-MDS}
Fano varieties are Mori Dream Spaces in the sense of Hu--Keel.
\end{thm}

The meaning and impact of these results will be discussed in the next section.

\section{Mori theoretic point of view}\label{sect:Mori}

Let $X$ be a normal projective variety and consider the dual $\R$-vector spaces
$$
N_1(X) := (Z_1(X)/\equiv)\otimes \R
\ \and\
N^1(X) := (\Pic(X)/\equiv)\otimes\R,
$$
where $\equiv$ denotes numerical equivalence.
The {\it Mori cone} of $X$ is the closure $\CNE(X) \subset N_1(X)$
of the cone spanned by classes of effective curves. Its dual cone
is the {\it nef cone} $\Nef(X) \subset N^1(X)$, which by Kleiman's criterion
is the closure of the cone spanned by ample classes.
The closure of the cone spanned by effective classes in $N^1(X)$
is the {\it pseudo-effective cone} $\PEff(X)$. Sitting in between
the nef cone and the pseudo-effective cone is the {\it movable cone
of divisors}
$\Mov(X)$, given by the closure of the cone spanned by classes of divisors
moving in a linear system with no fixed components.
All of these cones
$$\Nef(X) \subset \Mov(X)\subset\PEff(X)\subset N_1(X)$$
carry important geometric information about the variety $X$.

The Cone Theorem says that $\CNE(X)$ is generated
by the set of its $K_X$ positive classes
$\CNE(X)_{{K_X}_{\geq 0}}=\{\a \in \CNE(X)| K_X\. \a \geq 0\}$
and at most countably many $K_X$ negative rational
curves $C_i\subset X$ of  bounded anti-canonical degree
$0< - K_X\cdot C_i\leq 2\dim (X)$.
In particular the only accumulation points for the curve classes
$[C_i]\in \CNE (X)$ are along the hyperplane determined by $K_X\cdot \a =0$.
In particular, for a Fano variety, the Mori cone $\CNE (X)$ is
a rational polyhedral cone.
By duality, it follows that the nef cone $\Nef(X)=(\CNE (X))^\vee$
is also a rational polyhedral cone.

The geometry of $X$ is reflected to a large extent in the combinatorial
properties of $\CNE (X)$.
Every extremal face $F$ of $\CNE(X)$ corresponds to a surjective morphism
$\cont_F \colon X \to Y$, which is called a {\it Mori contraction}.
The morphism $\cont_F $  contracts precisely those curves on $X$
with class in $F$. Conversely, any morphism with connected fibers onto
a normal variety arises in this way.

\begin{rmk}
When $X$ is not a Fano variety,
$\CNE (X)_{{K_X}_{<0}}$ may fail to be finitely generated, and even in
very explicit
examples such as blow-ups of $\P^2$, the structure of the $K_X$ positive part of
the Mori cone is in general unknown.
Consider, for example, the long-standing
open conjectures of Nagata and Segre--Harbourne--Gimigliano--Hirschowitz.
\end{rmk}

A similar behavior, that we will now describe,
also occurs for the cone of nef curves.
By definition the cone of nef curves $\CNM(X)\subset N_1(X)$ is the closure of
the cone
generated by curves belonging to a covering family (i.e. to a family of curves
 that dominates the variety $X$). It is clear that if $\a\in \CNM (X)$ and
$D$ is an effective Cartier divisor on $X$, then
$\a\cdot D\geq 0$. It follows that $\a\cdot D\geq 0$ for any pseudo-effective
divisor $D$ on $X$. The remarkable result of \cite{BDPP} is the following.

\begin{thm}
The cone of nef curves is dual to the cone of pseudo-effective
divisors, i.e. $ \CNM (X)=\PEff(X)^\vee $.
\end{thm}

We now turn our attention to the case of $\mathbb Q$-factorial
Fano varieties. In this case, the cone of nef curves $\CNM (X)$
is also rational polyhedral and every extremal ray corresponds
to a {\it Mori fiber space} $X' \to Y'$ on a model $X'$
birational to $X$. More precisely we have
the following result (cf. \cite[1.3.5]{BCHM}).

\begin{thm}
$R$ is an extremal ray of $\CNM (X)$ if and only if there exists a
$\mathbb Q$-divisor $D$ such that $(X,D)$ is Kawamata log
terminal, and a $(K_X+D)$ Minimal Model Program $X\dasharrow X'$
ending with a Mori fiber space $X'\to Y'$, such that
the numerical transform of any curve in the fibers of $X' \to Y'$
(e.g., the proper transform of a general complete intersection
curve on a general fiber of $X' \to Y'$) has class in $R$.
\end{thm}

We will refer to the induced rational map $X \rat Y'$ as
a {\it birational Mori fiber structure} on $X$.
This was first studied by Batyrev \cite{Bat} in dimension three.
The picture in higher dimensions was recently established
by Birkar--Cascini--Hacon--M$^{\rm c}$Kernan for Fano varieties
and, in a more general context, independently by Araujo and Lehmann \cite{Ara, Leh}.
As a side note, even if it is known that the fibers of any Mori fibration
$X' \to Y'$ are covered by rational curves, it still remains an open question
whether the extremal rays of $\CNM(X)$ are spanned by classes of
rational curves.
This is related to a delicate question
on the rational connectivity of the smooth locus of
singular varieties.

The dual point of view (looking at $N^1(X)$
rather than $N_1(X)$), also offers a natural way of
refining the above results.
As mentioned above, if $X$ is a $\mathbb Q$-factorial
Fano variety, then it is a {\it Mori Dream Space} (cf. \cite{HK,BCHM}).
The movable cone $\Mov(X)$ of a {Mori Dream Space}
admits a finite decomposition into rational polyhedral cones,
called {\it Mori chambers}. One of these chambers is the
nef cone of $X$. The other chambers are given by nef cones
of $\mathbb Q$-factorial birational models $X' \sim_{\rm bir} X$ which are isomorphic
to $X$ in codimension one. Note indeed that any such
map gives a canonical isomorphism between $N^1(X)$ and $N^1(X')$.
Wall-crossings between contiguous Mori chambers correspond
to flops (or flips, according to the choice of the log pair structure)
between the corresponding birational models.
We can therefore view the Mori chamber decomposition of $\Mov(X)$
as encoding information not only on the biregular structure of $X$
but on its birational structure as well.

There is a way of recovering all this information from
the total coordinate ring, or Cox ring, of
a Mori Dream Space $X$, via a GIT construction.
For simplicity, we assume that the map $\Pic(X) \to N^1(X)$
is an isomorphism and that the class group of Weil divisors $\Cl(X)$ of $X$
is finitely generated. These properties hold if $X$ is a Fano variety.
The property that $\Pic(X) \cong N^1(X)$
simply follows by the vanishing of $H^i(X,\O_X)$ for $i > 0$.
The finite generation of $\Cl(X)$ is instead a deeper property;
a proof can be found in \cite{Tot}.
Specifically, see Theorem~3.1. in {\it loc.cit.},
which implies that the natural
map $\Cl(X) \to H_{2n-2}(X,\Z)$ is an isomorphism
for any $n$-dimensional Fano variety $X$.

A {\it Cox ring} of $X$ is, as defined in \cite{HK}, a ring of the type
$$
R(L_1,\dots,L_r) := \bigoplus_{m \in \Z^\r} H^0(X,\O_X(m_1L_1+ \dots + m_\r L_\r)),
$$
for any choice of line bundles $L_1,\dots,L_\r$
inducing a basis of $N^1(X)$.
Here $\r = \r(X)$ is the Picard number of $X$, and $m = (m_1,\dots,m_r)$.
We will call the {\it full Cox ring} of $X$ the ring
$$
R(X) := \bigoplus_{[D] \in \Cl(X)} H^0(X,\O_X(D)).
$$
If $X$ is factorial (that is, if the map $\Pic(X) \to \Cl(X)$ is
an isomorphism) and the line bundles $L_i$ induce a basis
of the Picard group, then the two rings coincide.

These rings were first systematically studied by
Cox in \cite{Cox} when $X$ is a toric variety.
If $X$ is a toric variety and $\D$ is the fan of $X$, then
the full Cox ring is the polynomial ring
$$
R(X) = \C[x_\l \mid \l \in \D(1)],
$$
where each $x_\l$ defines a prime toric invariant divisor of $X$.
When $X$ is smooth, this property characterizes toric varieties.
More precisely, Hu--Keel (cf. \cite{HK}) prove the following result.

\begin{thm}
Assume that $X$ is a smooth Mori Dream Space. Then $R(X)$ is
isomorphic to a polynomial ring
if and only if $X$ is a toric variety.
\end{thm}

More generally, Hu--Keel prove that a $\Q$-factorial Mori Dream Space $X$
can be recovered from any of its Cox rings via a GIT construction.
Moreover, the Mori chamber decomposition
of $X$ descends to $X$ via this construction from a
chamber decomposition associated to variations
of linearizations in the GIT setting.
From this perspective, the
Cox ring of a Fano variety is a very rich invariant,
encoding all essential information on the biregular and
birational geometry of the variety.

The above discussion shows how the main features of the geometry of a Fano variety $X$,
both from a biregular and a birational standpoint, are encoded in combinatorial
data embedded in the spaces $N_1(X)$ and $N^1(X)$.
Loosely speaking, we will say that that geometric properties of $X$
that are captured by such combinatorial data
constitute the {\it Mori structure} of $X$.

In the remaining part of the paper, we will discuss to which extend
the Mori structure of a Fano variety is preserved
under flat deformations. Any positive result in this direction
should be thought of as a rigidity statement.

The following result of Wi\'sniewski is the first strong evidence
that Fano varieties should behave in a somewhat rigid way under deformations
(cf. \cite{Wis,Wis2}).

\begin{thm}\label{thm:Wis}
The nef cone is locally constant in smooth families of Fano varieties.
\end{thm}

First notice that if $f \colon X \to T$ is a smooth family of Fano varieties,
then $f$ is topologically trivial,
and thus, if we denote by $X_t := f^{-1}(t)$ the fiber over $t$,
the space $N^1(X_t)$, being naturally isomorphic to
$H^2(X_t,\R)$, varies in a local system.
By the polyhedrality of the nef cone, this local system has finite monodromy.
This implies that, after suitable \'etale base change, one can
reduce to a setting where the spaces $N^1(X_t)$ are all naturally isomorphic.
The local constancy can therefore be intended in the \'etale topology.

Wi\'snienwki's result is the underlying motivation for
the results that will be discussed in the following sections.

\section{Deformations of the Cox rings}

The proof of Theorem~\ref{thm:Wis} has three main ingredients:
the theory of deformations of embedded rational curves,
Ehresmann Theorem,
and the Hard Lefschetz Theorem. All these ingredients use
in an essential way the fact that the family is smooth.
On the other hand, the very definitions involved in the whole Mori
structure of a Fano variety use steps in the Minimal Model Program,
which unavoidably generate singularities.
With this in mind, we will present a different approach
to the general problem of studying the deformation of Mori structures.
The main ingredients of this approach will be the use of
the Minimal Model Program in families, and
an extension theorem for sections of line bundles
(and, more generally, of divisorial reflexive sheaves).
The first implications of such approach will be on the Cox rings.
These applications will be discussed in this section.
Further applications will then presented in the following section.

When working with families of singular Fano varieties, one needs to be very cautious.
This is evident for instance in
the simple example of quadric surfaces degenerating to a quadric cone:
in this case allowing even the simplest surface singularity creates
critical problems (the Picard number dropping in the central fiber),
yielding a setting where the questions themselves cannot be posed.

We will restrict ourselves
to the smallest category of singularities which is preserved
in the Minimal Model Program, that of $\Q$-factorial terminal singularities.
This is the setting considered in \cite{dFH}.
As explained by Totaro \cite{Tot}, many of the results presented below
hold in fact under weaker assumptions on the singularities.

We consider a small flat deformation $f \colon X \to T$ of a Fano variety $X_0$.
Here $T$ is a smooth curve with a distinguished point $0 \in T$,
and $X_0 = f^{-1}(0)$. We assume that $X_0$ has terminal $\Q$-factorial singularities.
A proof of the following basic result can be found in
\cite[Corollary~3.2 and Proposition~3.8]{dFH}, where an
analogous but less trivial result is also proven to
hold for small flat deformations of weak log Fano varieties
with terminal $\Q$-factorial singularities.

\begin{prop}
For every $t$ in a neighborhood of $0$ in $T$, the fiber
$X_t$ is a Fano variety with terminal
$\Q$-factorial singularities.
\end{prop}

After shrinking $T$ near $0$, we can therefore assume that $f \colon X \to T$
is a flat family of Fano varieties with terminal $\Q$-factorial singularities.
If $t \in T$ is a general point, the monodromy on $N^1(X_t)$
has finite order. This can be seen using the fact that the monodromy
action preserves the nef cone of $X_t$, which is finitely generated
and spans that whole space. After a suitable base change, one may always reduce
to a setting where the monodromy is trivial.

If $f$ is a smooth family, then it is topologically trivial,
and we have already noticed
that the spaces $N^1(X_t)$ vary in a local system.
We have remarked how in general the dimension of these spaces may jump
if $f$ is not smooth.
Under our assumptions on singularities the property remains however true.
The proof of following result is given in \cite[Proposition~6.5]{dFH}, and
builds upon results of Koll\'ar--Mori \cite[(12.10)]{KM92}.

\begin{thm}\label{thm:N^1-const}
The spaces $N^1(X_t)$ and $N_1(X_t)$ form local systems on $T$ with finite monodromy.
After suitable base change, for every $t \in T$ there are natural isomorphisms
$N^1(X/T) \cong N^1(X_t)$ and $N_1(X_t) \cong N_1(X/T)$ induced, respectively,
by pull-back and push-forward.
\end{thm}

A similar property holds for the class group, and is stated next.
The proof of this property is given in \cite[Lemma~7.2]{dFH},
and uses the previous result in combination with a
generalization of the Lefschetz Hyperplane Theorem due to
Ravindra--Srinivas \cite{RS06}
(the statement is only given for toric varieties, but the proof
works in general).
As shown by Totaro in \cite[Theorem~4.1]{Tot}, the
same result holds more generally, only imposing that $X$ is a projective
variety with rational singularities and
$H^1(X,\mathcal O _X)= H^2(X,\mathcal O _X)=0$ (these conditions
hold for any Fano variety) and that $X_0$ is smooth in codimension
$2$ and  $\Q$-factorial in codimension three.

\begin{thm}
With the same assumptions as in Theorem~\ref{thm:N^1-const},
the class groups $\Cl(X_t)$ form a local system on $T$ with finite monodromy.
After suitable base change, for every $t \in T$ there are natural isomorphisms
$\Cl(X/T) \cong \Cl(X_t)$ induced
by restricting Weil divisors to the fiber (the restriction is
well-defined as the fibers are smooth in codimension one
and their regular locus is contained in the regular locus of $X$).
\end{thm}

For simplicity, we henceforth assume that the monodromy is trivial.
It follows by the first theorem that one can fix a common grading
for Cox rings of the fibers $X_t$ of the type considered in \cite{HK}.
The second theorem implies that
the there is also a common grading, given by $\Cl(X/T)$,
for the full Cox rings of the fibers.
This is the first step needed to control the Cox rings along the deformation.
The second ingredient is the following extension theorem.

\begin{thm}\label{thm:extension:Fano-case}
With the above assumptions, let $L$ be any Weil divisor on $X$ that
does not contain any fiber of $f$ in its support. Then, after possibly
restricting $T$ (and consequently $X$) to a neighborhood of $0$,
the restriction map
$$
H^0(X,\O_X(L)) \to H^0(X_0,\O_{X_0}(L|_{X_0}))
$$
is surjective (here $L|_{X_0}$ denotes the restriction of $L$ to $X_0$
as a Weil divisor).
\end{thm}

When $L$ is Cartier, this theorem is a small generalization
of Siu's invariance of plurigenera for varieties of general type.
The formulation for Weil divisors follows by a more general
result of \cite{dFH}, which is recalled below in
Theorem~\ref{thm:extension:general-case}.

As a corollary of the above theorems, we obtain the flatness of the Cox rings.

\begin{cor}
The full Cox ring $R(X_0)$ of $X_0$, as well as
any Cox ring $R(L_{0,1},\dots,L_{0,\r})$ of $X_0$
(provided the line bundles $L_{0,i}$ on $X_0$ are sufficiently divisible),
deform flatly in the family.
\end{cor}

The flatness of the deformation of the full Cox ring has a very interesting
consequence when applied to deformations of toric Fano varieties.

\begin{cor}
Simplicial toric Fano varieties with terminal singularities are rigid.
\end{cor}

The proof of this corollary is based on the simple observation that
a polynomial ring has no non-isotrivial flat deformations.
This theorem appears in \cite{dFH}.
When $X$ is smooth, the result was already known, and follows by
a more general result of Bien--Brion on the vanishing of $H^1(X,T_X)$
for any smooth projective variety admitting a toroidal embedding
(these are also known as {\it regular varieties}).
The condition that the toric variety is simplicial is the translation,
in toric geometry, of the assumption of $\Q$-factoriality.
The above rigidity result holds in fact more in general,
only assuming that the toric Fano variety is smooth in codimension $2$ and
$\Q$-factorial in codimension 3. This was proven by Totaro in \cite{Tot}
using the vanishing theorems of Danilov and Mustata
$H^i(\tilde \Omega ^j\otimes \mathcal O (D)^{**})=0$
for $i>0$, $j>0$ and $D$ an ample Weil divisor on a projective toric variety.

\section{Deformations of the Mori structure}

The flatness of Cox rings in flat families of Fano varieties
with terminal $\Q$-factorial singularities
is already evidence of a strong rigidity property of such varieties.
In this section, we consider a flat family $f \colon X \to T$ of Fano varieties
with terminal $\Q$-factorial singularities, parameterized by a smooth curve $T$.

An immediate corollary of Theorem~\ref{thm:extension:Fano-case}
is the following general fact.

\begin{cor}\label{cor:PEff}
For any flat family $f \colon X \to T$ of Fano varieties
with terminal $\Q$-factorial singularities over a smooth curve $T$,
the pseudo-effective cones $\PEff(X_t)$ of the fibers of $f$
are locally constant in the family.
\end{cor}

If one wants to further investigate how the Mori structure varies in the family,
it becomes necessary to run the Minimal Model Program. This requires
us to step out, for a moment, from the setting of families of Fano varieties.

Suppose for now that $f \colon X \to T$ is just a flat projective
family of normal varieties with $\Q$-factorial singularities.
Let $X_0$ be the fiber over a point $0 \in T$.
We assume that the restriction map $N^1(X) \to N^1(X_0)$ is surjective,
and that there is an effective $\Q$-divisor $D$ on $X$, not containing $X_0$
in its support, such that $(X_0,D|_{X_0})$ is a Kawamata log terminal
pair with canonical singularities. Assume furthermore that
$D|_{X_0} - a K_{X_0}$ is ample for some $a > -1$.
Note that this last condition always holds for Fano varieties.

The following result is crucial for our investigation.

\begin{thm}\label{thm:MMP-over-T}
With the above notation,
every step $X^i\dasharrow X^{i+1}$ in the Minimal Model Program of
$(X,D)$ over $T$ with scaling of $D - a K_{X}$  is either trivial on the
fiber $X_0^i$ of $X^i$ over $0$, or it
induces a step of the same type (divisorial contraction, flip, or Mori fibration)
$X^i_0\dasharrow X^{i+1}_0$ in the Minimal Model Program of $(X_0,D|_{X_0})$
with scaling of $D|_{X_0} - a K_{X_0}$.
In particular, at each step $X_0^i$ is the proper transform of $X_0$.
\end{thm}

For a proof of this theorem, we refer the reader to \cite{dFH}
(specifically, see Theorem~4.1
and the proof of Theorem~4.5 {\it loc.cit}).
The key observation is that, by running a
Minimal Model Program  with scaling of $D - a K_{X}$,
we can ensure that the property that
$D|_{X_0} - a K_{X_0}$ is ample for some $a > -1$
is preserved after each step of the program.
By the semicontinuity of fiber
dimensions, it is easy to see that $X^i\dasharrow X^{i+1}$ is a Mori
fiber space if and only if so is $X^i_0\dasharrow X^{i+1}_0$.
If $X^i\dasharrow X^{i+1}$ is birational, then the main issue is to show that
if  $X^i\dasharrow X^{i+1}$ is a flip and $X^i\to Z^i$ is the corresponding
flipping contraction, then $X^i_0\to Z^i_0$ is also a flipping contraction.
If this were not the case, then $X^i_0\to Z^i_0$ would be a divisorial
contraction and hence $Z^i_0$ would be $\mathbb Q$-factorial.
Since $D^i|_{X_0^i}-aK_{X_0^i}$ is nef over $Z_0^i$, it follows that
$-K_{X_0^i}$
is ample over $Z_0^i$ and hence that $Z^i_0$ is canonical.
By \cite[Proposition~3.5]{dFH} it then follows that $Z^i$ is $\Q$-factorial.
This is the required contradiction as the target of a flipping
contraction is never  $\Q$-factorial.
Therefore it follows that $X^i\to Z^i$ is a flipping contraction if
and only if so is $X^i_0\to Z^i_0$.

\begin{rmk}
The theorem implies that
$X^i_0\dasharrow X^{i+1}_0$ is a divisorial contraction or a Mori fibration
if and only if so is $X^i_t\dasharrow X^{i+1}_t$ for general $t\in T$.
However, there exist flipping contractions  $X^i\dasharrow X^{i+1}$ which are
the identity on a general fiber $X^i_t$.
This follows from the examples of Totaro that we will
briefly sketch at the end of the section.
\end{rmk}

One of the main applications of this result is the following
extension theorem (cf. \cite[Theorem~4.5]{dFH}), which in particular implies the statement
of Theorem~\ref{thm:extension:Fano-case} in the case of families of Fano varieties.

\begin{thm}\label{thm:extension:general-case}
With the same notation as in Theorem~\ref{thm:MMP-over-T},
assume that either $D|_{X_0}$ or $K_{X_0} + D|_{X_0}$ is big.
Let $L$ be any Weil divisor whose support does not contain $X_0$
and such that $L \equiv k(K_X + D)|_{X_0}$ for some rational number $k > 1$.
Then the restriction map
$$
H^0(X,\O_X(L)) \to H^0(X_0,\O_{X_0}(L|_{X_0}))
$$
is surjective.
\end{thm}

There are versions of the above results where the condition on the
positivity of $D|_{X_0} - a K_{X_0}$ is replaced by the condition
that the stable base locus of $K_X + D$ does not contain any irreducible
component of $D|_{X_0}$ (cf. \cite[Theorem~4.5]{dFH}).
The advantage of the condition considered here is that it
only requires us to know something about the special fiber $X_0$.
This is a significant point, as
after all we are trying to lift geometric
properties from the special fiber to the whole space and nearby fibers
of an arbitrary flat deformation.

We now come back to the original setting, and hence
assume that $f \colon X \to T$ is a flat family of Fano varieties
with $\Q$-factorial terminal singularities.
After \'etale base change, we can assume that $N^1(X_t) \cong N^1(X/T)$
for every $t$.

Corollary~\ref{cor:PEff} implies, by duality,
that the cones of nef curves $\CNM(X_t)$ are constant in the family.
Combining this with Theorem~\ref{thm:MMP-over-T},
we obtain the following rigidity property of Mori fiber
structures.

\begin{thm}\label{thm:Mfs-over-T}
The birational Mori fiber structures $X_t\dasharrow X'_t\to Y'_t$ are locally
constant in the family $X\to T$.
\end{thm}

This result is implicit in \cite{dFH}. As it was not explicitly stated there,
we provide a proof.

\begin{proof}
Let $R_0$ be the extremal ray of $\CNM(X_0)$ corresponding to a given
birational Mori fiber structure on $X_0$.
Note that by \cite[1.3.5]{BCHM} and its
proof, there exists an ample $\mathbb R$-divisor $A_0$ such that the
$K_{X_0}$ Minimal Model Program with scaling of
$A_0$ say $X_0\dasharrow X'_0$ ends with
Mori fiber space $X'_0 \to Y'_0$ which is $(K_{X_0}+A_0)$-trivial.
Notice also that if we make a general choice of $A_0$ in $N^1(X_0)$, then each step
of this Minimal Model Program with scaling is uniquely determined since at each
step there is a unique $K_{X_0}+t_iA_0$ trivial extremal ray.

We may now assume that there is an ample $\mathbb R$-divisor $A$ on $X$ such
that $A_0=A|_{X_0}$. Consider running the $K_X$ Minimal
Model Program over $T$ with scaling of $A$ say $X\dasharrow X'$.
Since $X$ is uniruled, this ends with a Mori fiber space $X' \to Y'$.
By Theorem~\ref{thm:MMP-over-T}, this induces the
Minimal Model Program with scaling on the fiber $X_0$ considered
in the previous paragraph.
Moreover, the Minimal Model Program on $X$ terminates with the Mori fiber space
$X' \to Y'$ at the same step when the induced
Minimal Model Program on $X_0$ terminates with the Mori fiber space
$X_0' \to Y_0'$. This implies that the
the birational Mori fiber structure $X_0 \rat Y_0'$ extends to
the birational Mori fiber structure $X \rat Y'$, and thus deforms
to a birational Mori fiber structure on the nearby fibers.
\end{proof}

A similar application of Theorems~\ref{thm:MMP-over-T}
and~\ref{thm:extension:general-case} leads to
the following rigidity result for the cone of moving divisors
(cf. \cite{dFH}).

\begin{thm}
The moving cone $\Mov(X_t)$ of divisors is locally constant in the family.
\end{thm}

\begin{proof}
The proof is similar to the proof of Theorem~\ref{thm:Mfs-over-T}
once we observe that the faces of ${\rm Mov}(X)$ are determined by
divisorial contractions and that given an extremal contraction $X\to Z$ over $T$,
this is divisorial if and only if the contraction on the central
fiber $X_0\to Z_0$ is divisorial.
\end{proof}

Regarding the behavior of the nef cone and, more generally,
of the Mori chamber decomposition of the moving cone,
the question becomes however much harder. In fact, once we allow even the
mildest singularities,
the rigidity of the
whole Mori structure only holds in small dimensions.
The following result was proven in \cite[Theorem~6.9]{dFH}.

\begin{thm}
With the above notation, assume that $X_0$ is either at most 3-dimensional,
or is 4-dimensional and Gorenstein.
Then the Mori chamber decomposition of $\Mov(X_t)$ is locally
constant for $t$ in a neighborhood of $0 \in T$.
\end{thm}

In \cite{Tot}, Totaro provides families of examples that show that
this result is optimal.
In particular he shows that there is a family of terminal $\mathbb Q$-factorial
Gorenstein Fano varieties $X\to T$ such that $X_t\cong \mathbb P ^3\times \mathbb P ^2$
for $t\ne 0$ and ${\rm Nef }(X_0)$ is strictly contained in
${\rm Nef }(X_t)$. The reason for this is that there is a flipping contraction
$X\to Z$ over $T$ which is an isomorphism on the general fiber $X_t$ but contracts a copy
of $\mathbb P ^3$ contained in $X_0$.
Let $X^+\to Z$ be the corresponding flip and fix $H^+$ a divisor on $X^+$ which is ample
over $T$. If $H$ is its strict transform on $X$, then $H$ is negative on flipping curves
and hence $H|_{X_0}$ is not ample, however $H|_{X_t}\cong H^+|_{X^+_t}$ is ample for
$t\ne 0$. Therefore, the nef cone of $X_0$ is strictly smaller than the nef cone of $X_t$
so that the Mori chamber decomposition of $\Mov(X_0)$ is finer than that of $\Mov (X_t)$.
The construction of this example starts from the flip
from the total space of $\O_{\P^3}(-1)^{\oplus 3}$
to the total space of $\O_{\P^2}(-1)^{\oplus 4}$.
The key idea is to interpret this local setting
in terms of linear algebra, and use such description
to compactify the setting into a family of Fano varieties.
Totaro also gives an example in dimension 4, where the generic
element of the family is isomorphic to the blow-up of $\P^4$
along a line, and the central fiber is a Fano variety with
$\Q$-factorial terminal singularities that is not Gorenstein.

\begin{rmk}
The fact that the Mori chamber decomposition is not in general locally
constant in families of Fano varieties with $\Q$-factorial terminal singularities
is not in contradiction with the flatness of Cox rings.
The point is that the flatness of such rings is to be understood only as modules,
but it gives no information on the ring structure. The changes in the Mori chamber
decomposition are related to jumps of
the kernels of the multiplication maps.
\end{rmk}

\providecommand{\bysame}{\leavevmode \hbox \o3em
{\hrulefill}\thinspace}

\end{document}